\documentclass[11pt,a4paper]{article}
\usepackage{graphicx}
\usepackage{url}
\usepackage{amsmath, amsthm, amssymb}
\usepackage{float}
\usepackage[multiple]{footmisc}
\usepackage{authblk}
\usepackage{fullpage}

\newlength{\bibitemsep}
\newlength{\bibparskip}
\let\oldthebibliography\thebibliography
\renewcommand\thebibliography[1]{%
  \oldthebibliography{#1}%
  \setlength{\parskip}{0pt}%
  \setlength{\itemsep}{0pt plus 0.1ex}%
}

\numberwithin{equation}{section}

\renewcommand\baselinestretch{1.2}

\newtheorem{thm}{Theorem}[section]
\newtheorem{cor}[thm]{Corollary}
\newtheorem{lem}[thm]{Lemma}

\theoremstyle{definition}

\theoremstyle{definition}

\newenvironment{eq}
{ \begin{equation} 
  }
{ \end{equation}     } 

\newenvironment{ew}
{ \begin{equation*} 
  }
{ \end{equation*}     }

\begin{document}

\title{On Kummer's test of convergence and its relation to basic comparison tests}
\author{Frantisek Duris\thanks{fduris@dcs.fmph.uniba.sk}}
\affil{Department of Computer Science, \\ Faculty of Mathematics, Physics and Informatics, \\ Comenius University, \\ Mlynska dolina 842 48 Bratislava}

\maketitle

\begin{center}
\textbf{Abstract}
\end{center}

\begin{center}
\begin{minipage}{0.85\linewidth}
Testing convergence of infinite series is an important part of mathematics. A very basic test of convergence is to upper-bound a given series with a known series, term by term. In $19^{th}$ century, Kummer proposed a test of convergence for any positive series based on finding a suitable positive sequence $\{p_n\}$ and a suitable real constant $c$. It can be easily shown that by choosing appropriate sequence $\{p_n\}$, the Kummer's test yields other tests like Raabe's, Gauss' or Bertrand's as its special cases. In 1995, Samelson noted that there is another interesting relation between Kummer's test and basic comparison tests, particularly, that one can transform the sequence $\{p_n\}$ into a convergent bounding series, and he sketched a simple proof of this statement. In this paper, we fill the missing formal proof, although using a different approach, and we show how to construct a bounding series from the sequence $\{p_n\}$ and vice versa.
\end{minipage}
\end{center}

~

~

~\vfill
\eject

\section{Introduction}

In the theory of infinite series, the Kummer's test of convergence/divergence has an important place because it can determine the character of any positive series. This test was first given by German mathematician Ernst Kummer in 1835 \cite{kummer1835convergenz}, and it was later improved by Ulisse Dini \cite{dini1867sulle}. More recently, a short and simple proof of this test was given by Tong \cite{tong1994kummer}. Briefly, the test asserts that the series $\sum a_n$ converges if and only if there is a positive series $\sum p_n$ and a real constant $c>0$ such that $p_n(a_n/a_{n+1} - p_{n+1})\geq c$ for all but a finite number of $n$. The power of Kummer's test is emphasized by the fact that other convergence tests like Raabe's, Gauss' or Bertrand's can be viewed as its special cases obtained by choosing specific values of $p_n$'s. 

The comparison of Kummer's test with the basic comparison test (introduced explicitly by Cauchy in 1821, see \emph{Section \ref{sec:elem}} for more details) first appeared in Samelson \cite{samelson1995kummer}. Samelson's proof sketch is based on the fact that all positive series $\sum b_n$ can be written in the form $\sum (c_{n-1} - c_n)$, where $c_n = \sum _{n+1}^{\infty} b_i$. In our paper, we give the proof of the same statement, namely that \emph{Kummer's test is equal to the basic comparison test}, by using a different approach. More particularly, we show how to construct a bounding series from the the comparison test from the numbers $\{p_n\}$ and vice versa. Additionally, because Samelson only sketched his proof, we fill the gap by providing a full formal proof of this statement.

For the sake of completeness, we state a few basic theorems from the theory of infinite series (\emph{Section \ref{sec:elem}}) as well as proofs of two lemmas used in the main proof (\emph{Section \ref{sec:kummer}}). Note that, unless stated otherwise, the expression for a positive series $\sum_{n=1}^{\infty} a_n$ will be abbreviated as $\sum a_n$. 

\section{Elementary theorems} \label{sec:elem}

In this section, we provide the most basic theorems in convergence of positive series which will be referenced in the next section.

\begin{thm}[Cauchy-Bolzano]
The series $\sum a_n$ converges if and only if 
\begin{ew}
\forall \epsilon > 0,~ \exists N\in\mathbb{N},~ \forall n \in \mathbb{N}~ n>N,~ \forall p\in\mathbb{N} :~
|a_n + a_{n+1} + ... + a_{n+p}| < \epsilon.
\end{ew}
\label{thm:CB}
\end{thm}

\begin{cor}
If series $\sum a_n$ converges, then $\lim_{n\to\infty}a_n=0$.
\label{cor:CB1}
\end{cor}

\begin{cor}
If series $\sum a_n$ converges, then $\forall \epsilon>0,~ \exists k\in\mathbb{N}:~ \left| \sum_{n=k+1}^{\infty} a_n \right| < \epsilon$.
\label{cor:CB2}
\end{cor}

\begin{thm}[Basic comparison test 1]
Let $\sum a_n$ and $\sum b_n$ be two series with positive terms. Let at most finite count of numbers $n\in\mathbb{N}$ fail the inequality $a_n \leq b_n$. Then, the convergence of series $\sum b_n$ implies the convergence of series $\sum a_n$, and the divergence of series $\sum a_n$ implies the divergence of series $\sum b_n$.
\label{thm:BCT1}
\end{thm}

\begin{thm}[Basic comparison test 2]
Let $\sum a_n$ and $\sum b_n$ be two series with positive terms. Let at most finite count of numbers $n\in\mathbb{N}$ fail the inequality ${a_{n+1}}/{a_n} \leq {b_{n+1}}/{b_n}$. Then, the convergence of series $\sum b_n$ implies the convergence of series $\sum a_n$, and the divergence of series $\sum a_n$ implies the divergence of series $\sum b_n$.
\label{thm:BCT2}
\end{thm}

\section{Kummer's equivalence with basic comparison tests} \label{sec:kummer}

To demonstrate the equivalence of Kummer's test with the basic convergence tests from \emph{Section \ref{sec:elem}}, we first state two necessary lemmas (adopted from the exercises from \cite{kubacek}).

\begin{lem}
If $\sum a_n$ is a convergent series with positive terms, then there exists a monotonous sequence $\{B_n\}_{n=1}^{\infty}$ such that $\lim_{n\to\infty}B_n = \infty$ and $\sum a_nB_n$ converges.
\label{lem:kum1}
\end{lem}

\begin{proof}
Let $\sum a_n$ be a convergent series with positive terms. From \emph{Corollary \ref{cor:CB1}} it follows that $\lim_{n\to\infty}{1}/{a_n}=\infty$, and from \emph{Corollary \ref{cor:CB2}} we construct an increasing subsequence of natural numbers $\{\xi_n\}_{n=1}^{\infty}$ such that 
\begin{eq}
\left(\sum_{k>\xi_n} a_k\right) \frac{1}{a_n} < a_n.
\label{eq:1}
\end{eq}
We define the numbers $B_n$ as follows:
\begin{ew}
B_n =
\begin{cases}
    0  & : \text{if } n\in[1,\xi_1]\\
    \frac{1}{a_k} & : \text{if } n\in[\xi_k,\xi_{k+1}).
\end{cases}
\end{ew}
The convergence of $\sum a_nB_n$ follows from (\ref{eq:1}):
\begin{ew}
\sum_{n=1}^{\infty}a_nB_n = \sum_{n=1}^{\xi_1-1}a_n0 + \sum_{n=1}^{\infty}\left(\frac{1}{a_n}\sum_{k=\xi_n}^{\xi_{n+1}-1}a_k\right)
\leq \sum_{n=1}^{\infty}a_n.
\end{ew}
\end{proof}

\begin{lem}
If $\sum a_n$ is a divergent series with positive terms, then there exists a monotonous sequence $\{B_n\}_{n=1}^{\infty}$ such that $\lim_{n\to\infty}B_n = 0$ and $\sum a_nB_n$ diverges.
\label{lem:kum2}
\end{lem}

\begin{proof}
Let $\sum a_n$ be a divergent series with positive terms. We will consider two cases. First, if $\limsup_{n\to\infty} a_n \geq \epsilon > 0$, then by leaving out all terms $a_n<{\epsilon}/{2}$ we will not change the divergent character of the series. Let $\{\xi_k\}_{n=1}^{\infty}$ denote all indices $n$ such that $a_n\geq{\epsilon}/{2}$ (in the increasing order). We define the numbers $B_n$ as follows:
\begin{ew}
B_n =
\begin{cases}
    0  & : n\neq\xi_k \text{ for all } k\\
    \frac{1}{k} & : n=\xi_k \text{ for some } k.
\end{cases}
\end{ew}
Obviously, $\lim_{n\to\infty}=0$, and
\begin{ew}
\sum_{n=1}^{\infty}a_nB_n \geq \sum_{n=1}^{\infty}\frac{\epsilon}{2n} = \infty.
\end{ew}
Second, if $\lim_{n\to\infty}a_n=0$, then from the divergence of the series $\sum a_n$ we have an increasing subsequence $\{\xi_n\}_{n=1}^{\infty}$ of natural numbers such that 
\begin{eq}
\sum_{k=\xi_n}^{\xi_{n+1}-1}a_k > n, ~~ \xi_1=1.
\label{eq:2}
\end{eq}
We define the numbers $B_n$ as follows:
\begin{ew}
B_n = \frac{1}{k} \text{ for } n\in[\xi_k,\xi_{k+1}) \text{ and } k\in\mathbb{N}
\end{ew}
By (\ref{eq:2}) the series $\sum a_nB_n$ diverges because
\begin{ew}
\sum_{n=1}^{\infty}a_nB_n = \sum_{n=1}^{\infty}\left(\frac{1}{n}\sum_{k=\xi_n}^{\xi_{n+1}-1}a_k\right) \geq 
\sum_{n=1}^{\infty}1 = \infty.
\end{ew}
\end{proof}

Now we state and prove Kummer's test and thus demonstrate is equality with comparisons tests. The statement of the theorem is adopted from Tong \cite{tong1994kummer}.

\begin{thm}[Kummer's test \cite{tong1994kummer}]
Let $\sum a_n$ be a positive series.

(1) $\sum a_n$ is convergent if and only if there is a positive series $\sum p_n$ and a real number $c>0$, such that $p_n(a_n/a_{n+1}) - p_{n+1}\geq c$.

(2) $\sum a_n$ is divergent if and only if there is a positive series $\sum p_n$ and a real number $c>0$, such that $\sum 1/p_n$ diverges and $p_n(a_n/a_{n+1}) - p_{n+1}\leq 0$.
\end{thm}

\begin{proof}
\emph{Sufficiency.}

(1) Suppose that $\sum p_n$ is a positive series and there is a real number $c>0$, such that  $p_n(a_n/a_{n+1}) - p_{n+1}\geq c$. The inequalities 
\begin{ew}
p_na_n - p_{n+1}a_{n+1}\geq ca_{n+1} >0
\end{ew}
implies that the sequence $\{p_na_n\}_{n=1}^{\infty}$ is positive and decreasing; therefore, it has a limit. Furthermore, we can construct sequence of numbers $\{B_n\}_{n=1}^{\infty}$, $B_n \geq 1$ for all $n$, such that
\begin{ew}
p_na_n - p_{n+1}a_{n+1} = cB_{n+1}a_{n+1}.
\end{ew}
Thus, the series 
\begin{ew}
\sum_{n=1}^{\infty}a_{n+1}B_{n+1} = \frac{1}{c}\sum_{n=1}^{\infty}\left(p_na_n - p_{n+1}a_{n+1}\right) = \frac{p_1a_1}{c} - \frac{1}{c}\lim_{n\to\infty}p_na_n>0
\end{ew}
converges, and because $B_n \geq 1$, so does the series $\sum a_n$ according to the \emph{First comparison test (Theorem \ref{thm:BCT1})}.

(2) Suppose that $\sum p_n$ is a positive series for which $\sum 1/p_n$ diverges and $p_n(a_n/a_{n+1}) - p_{n+1} < 0$. This inequality immediately yields
\begin{ew}
\frac{1/p_{n+1}}{1/p_{n}} \leq \frac{a_{n+1}}{a_{n}},
\end{ew}
and since $\sum 1/p_n$ diverges, so does $\sum a_n$ according to \emph{Second comparison test (Theorem \ref{thm:BCT2})}.\\
\eject
\emph{Necessity.}

(1) Suppose that $\sum a_n$ is a convergent positive series. From \emph{Lemma \ref{lem:kum1}} we have a positive monotonous sequence $\{B_n\}_{n=1}^{\infty}$, such that $\lim_{n\to\infty} B_n=\infty$ and $\sum a_nB_n$ converges. Let $p_1$ be a positive number such that $\sum a_nB_n = p_1a_1$. We define the sequence $\{p_n\}_{n=1}^{\infty}$ recursively as follows:
\begin{ew}
p_{n+1}a_{n+1} = p_na_n - a_{n+1}B_{n+1}.
\end{ew}
The numbers $p_n$ are positive because
\begin{ew}
\lim_{n\to\infty}p_na_n = p_1a_1 - \lim_{n\to\infty}\sum_{k=1}^n a_{k+1}B_{k+1} = 0,
\end{ew}
and the numbers $\{a_n\}$ are positive as well. Furthermore, for any $c>0$ and sufficiently large $n$ we have $p_{n}a_{n} - p_{n+1}a_{n+1} = a_{n+1}B_{n+1} \geq a_{n+1}c$.

(2) Suppose that $\sum a_n$ is a divergent positive series. From \emph{Lemma \ref{lem:kum2}} we have a positive monotonous sequence $\{B_n\}_{n=1}^{\infty}$, such that $\lim_{n\to\infty} B_n=0$ and $\sum a_nB_n$ diverges. Let $p_n = 1/a_nB_n$. It follows that 
\begin{ew}
\frac{a_{n+1}B_{n+1}}{a_nB_n} = \frac{\frac{1}{p_{n+1}}}{\frac{1}{p_n}} \leq \frac{a_{n+1}}{a_n}
\end{ew}
which, according to the \emph{Second comparison test (Theorem \ref{thm:BCT2})}, implies the divergence of the positive series $\sum {1}/{p_n}$, as well as the inequality $p_na_n - p_{n+1}a_{n+1} \leq 0$.
\end{proof}

\addcontentsline{toc}{section}{References}
\renewcommand\baselinestretch{1}
\bibliographystyle{abbrv}
\bibliography{bibi}

\end{document}